\documentclass[11pt,a4paper]{amsart}

\usepackage[utf8]{inputenc}
\usepackage{latexsym}
\usepackage{amssymb}
\usepackage{amsmath}
\usepackage{tikz}

\title{Small deviations in lognormal Mandelbrot cascades}

\author[M. Nikula]{Miika Nikula}
\address{University of Helsinki, Department of Mathematics and Statistics,
         P.O. Box 68 , FIN-00014 University of Helsinki, Finland}
\email{miika.nikula@helsinki.fi}

\date{\today}

\newtheorem{parg}{§}
\newtheorem{defn}[parg]{Definition}
\newtheorem{thm}[parg]{Theorem}
\newtheorem{prop}[parg]{Proposition}

\newtheorem{cor}[parg]{Corollary}

\newcommand{\N}{\mathbb{N}}

\newcommand{\R}{\mathbb{R}}

\newcommand{\E}{\mathbb{E}}

\newcommand{\Prob}{\mathbb{P}}

\newcommand{\ind}{\mathbf{1}}

\newcommand{\Bcal}{\mathcal{B}}

\renewcommand{\d}{\, \mathrm{d}}

\newcommand{\eqlaw}{\stackrel{d}{=}}

\begin{document}

\maketitle

\begin{abstract}
We study small deviations in Mandelbrot cascades and some related models. Denoting by $Y$ the total mass variable of a Mandelbrot cascade generated by $W$, we show that if
$$
\lim_{x \to 0} \frac{\log \log 1/\Prob(W \leq x)}{\log \log 1/x} = \gamma > 1,
$$
then the Laplace transform of $Y$ satisfies
$$
\lim_{t \to \infty} \frac{\log \log 1/\E e^{-t Y}}{\log \log t} = \gamma.
$$
As an application, this gives new estimates for $\Prob(Y \leq x)$ for small $x > 0$. As another application of our methods, we prove a similar result for a variable arising as a total mass of a lognormal $\star$-scale invariant multiplicative chaos measure.
\end{abstract}

\section{Introduction}

We start by outlining the problem studied in this note. The Mandelbrot cascade on a binary tree is the following construction. Let $W$ be a given positive\footnote{By \emph{positive} we mean $\Prob(W > 0) = 1$.} random variable such that $\E W = 1/2$ and $\{W_\sigma\}_{\sigma \in \Sigma}$ be an i.i.d. collection of copies of $W$ indexed by the infinite binary tree $\Sigma = \bigcup_{n \geq 1} \{0,1\}^n$. The $n$-th cascade variable\footnote{To all $n \in \N$ one may also associate a random measure $M_n$ on $[0,1]$ by giving the dyadic interval naturally encoded by $\sigma \in \{0,1\}^n$ the mass $W_{\sigma_1} W_{\sigma_1 \sigma_2} \dots W_{\sigma_1 \sigma_2 \dots \sigma_n}$. However, we will not work with these cascade measures and mention them only for motivation.} $Y_n$ is defined by
$$
Y_n = \sum_{\sigma_1 \sigma_2 \dots \sigma_n \in \{0,1\}^n} W_{\sigma_1} W_{\sigma_1 \sigma_2} \dots W_{\sigma_1 \sigma_2 \dots \sigma_n}.
$$
The sequence $(Y_n)_{n \geq 1}$ is a positive martingale and as such almost surely convergent to a limit variable $Y = \lim_{n \to \infty} Y_n$ which satisfies the functional equation
\begin{equation}\label{eq:smoothing-transform}
Y \eqlaw W_0 Y_1 + W_1 Y_1,
\end{equation}
where the variables $W_0$, $W_1$, $Y_0$ and $Y_1$ are independent and $W_0 \eqlaw W_1 \eqlaw W$ and $Y_0 \eqlaw Y_1 \eqlaw Y$. A random variable $Y$ for which \eqref{eq:smoothing-transform} holds is called a \emph{fixed point of the smoothing transform associated to $W$}. For a given positive random variable $W$, the fixed points of the smoothing transform have been characterized by Durrett and Liggett \cite{duli83}. The specifics of the characterization do not concern our study here, but the following facts are good to know. First, if the smoothing transform given by $W$ has fixed points of finite mean, all the fixed points are given by constant multiples of the Mandelbrot cascade associated to $W$, and in the case of infinite mean the Mandelbrot cascade may (under rather general assumptions) be deterministically renormalized in order to obtain the fixed points \cite{aish11,we11,ma11}. Second, for a positive $W$ the fixed points of the smoothing transform are also positive, i.e. have $\Prob(Y > 0)$.

The study of the tail of the fixed points of smoothing transforms at positive infinity has a long history. Indeed, the finiteness of the moments $\E Y^p$, $p > 1$, was a central question already in the work of Mandelbrot \cite{ma74} on his cascades, answered by Kahane and Peyri\`ere \cite{kape76}. The behavior of the Laplace transform $\E e^{-tY}$ near $0$ was used by Durrett and Liggett in the characterization of the fixed points of the smoothing transforms.

Given that $\Prob(W > 0) = 1 = \Prob(Y > 0)$, the asymptotics of the probabilities of $Y$ being small are also of interest. To mention just some work on this question, in connection with the multifractal analysis of Mandelbrot cascade measures it was shown by Holley and Waymire \cite{howa92} that if there exists an $a > 0$ such that $\Prob(W \geq a) = 1$, the Laplace transform of cascade variable $Y$ satisfies $\E e^{-t Y} \leq \exp\left(-c t^b\right)$ for some constants $c > 0$ and $0 < b < 1$ depending on the distribution of $W$. In a more general study of multifractal analysis of Mandelbrot cascade measures, Molchan \cite{mo96} proved that if $\E W^{-q} < \infty$ for some $q > 0$, then also $\E Y^{-2q} < \infty$. These results have been later improved by Liu \cite{li01} and most recently Hu \cite{hu13}, who have shown much stronger results relating the asymptotics of $\Prob(W \leq x)$ near $0$ to the asymptotics of $\E e^{-t Y}$ near $\infty$ in the more general case of a smoothing transform in which the number of summands $W_i Y_i$ appearing on the right hand side of \eqref{eq:smoothing-transform} is random and the i.i.d. assumption on the $(W_i)$ is relaxed.

However, there is little earlier work on the asymptotics of $\E e^{-t Y}$ near $\infty$ in the important special case where the generator $W$ is lognormal: the best result in the literature seems to be Molchan's result on finiteness of the moments of negative order. The lognormal generator was considered already in the original work of Mandelbrot \cite{ma72}, but interest in specifically lognormal cascades has been revitalized recently by the analogies between lognormal Mandelbrot cascades and lognormal multiplicative chaos, a much more general construction of a random measure given originally by Kahane \cite{ka85}. Kahane's lognormal multiplicative chaos measures have recently been connected to problems in mathematical physics involving the exponential of the Gaussian free field, Liouville quantum gravity and the KPZ formula; we refer the reader to the recent survey of Rhodes and Vargas \cite{rhva13} for details on these connections.

Finally, we mention the work of Ostrovsky \cite{os09} in which a prediction is made for the exact form of the Mellin transform of the law of the total mass of a certain lognormal multiplicative chaos measure. While the accuracy of the asymptotics given by our result is certainly far from providing a rigorous proof to Ostrovsky's formula, it is worth noting that our result is in accordance with the prediction.

\vspace{0.3cm}

Our main result is Theorem \ref{thm:upper-bound} below, which connects the asymptotics of $\Prob(W \leq x)$ as $x \to 0$ for a class of random variables (that includes the lognormal variables) to the asymptotics of $\E e^{-t Y}$ as $t \to \infty$. This theorem is then applied to fixed points of the smoothing transform (Theorem \ref{thm:cascades}) and thus to (possibly renormalized) Mandelbrot cascades, and then to a simple example of multiplicative chaos measures (Theorem \ref{thm:chaos}).

The notation $X \preceq Z$ means that $X$ is stochastically dominated by $Z$, i.e. that $\Prob(X \geq x) \leq \Prob(Z \geq x)$ for all $x \in \R$.

\begin{thm}\label{thm:upper-bound}
Let $W$ and $Y$ be positive random variables satisfying
\begin{equation}\label{eq:smoothing-inequality}
Y \succeq W_0 Y_0 + W_1 Y_1,
\end{equation}
where $(Y_0,Y_1)$ is an independent pair of copies of $Y$, independent of $(W_0,W_1)$, and $W_0 \eqlaw W_1 \eqlaw W$. Suppose further that there exist $\gamma > 1$ and $x' \in ]0,1[$ such that
\begin{equation}\label{eq:w-small-tail}
\Prob(W \leq x) \leq \exp\left( - c (-\log x)^\gamma \right) \quad \textrm{for all } x \leq x'.
\end{equation}
Then for any $\alpha \in [1,\gamma[$ there exists a constant $t_\alpha > 0$ such that for all $t \geq t_\alpha$ we have
\begin{equation}\label{eq:main-estimate}
\E e^{-t Y} \leq \exp\left(- c_\alpha (\log t)^\alpha \right) \quad \textrm{for all } t \geq t_\alpha.
\end{equation}
\end{thm}
\noindent \emph{Remarks.} Note that $W_0$ and $W_1$ are not assumed to be independent. The assumption $W_0 \eqlaw W_1$ could easily be relaxed, in which case instead of the assumption \eqref{eq:w-small-tail} one would assume that the lighter of the negative tails of $\log W_0$ and $\log W_1$ would satisfy a similar condition. Similarly, instead of two summands $W_i Y_i$, $i=0,1$, we could consider an arbitrary but fixed finite number $N$ of summands.

\begin{thm}\label{thm:cascades}
Suppose $W$ is a positive random variable satisfying
$$
\lim_{x \to 0} \frac{\log \log 1/\Prob(W \leq x)}{\log \log 1/x} = \gamma > 1.
$$
Let $Y$ be the fixed point of the smoothing transform associated to $W$ or in other words the limit variable of the (possibly renormalized) Mandelbrot cascade generated by $W$, i.e. suppose that
$$
Y \eqlaw W_0 Y_0 + W_1 Y_1.
$$
Then
$$
\lim_{t \to 0} \frac{\log \log 1/\E e^{-tY}}{\log \log t} = \gamma.
$$
\end{thm}
We state the resulting estimate for $\Prob(Y \leq x)$ as a corollary.
\begin{cor}\label{cor:cascades}
Let $W$ and $Y$ be as in Theorem \ref{thm:cascades}. Then also
$$
\lim_{x \to 0} \frac{\log \log 1/\Prob(Y \leq x)}{\log \log 1/x} = \gamma.
$$
\end{cor}


\vspace{0.3cm}

\noindent In Section \ref{se:applications} we exhibit an example of a lognormal multiplicative chaos measure for which we state and prove a corresponding result.


\section{Proof of Theorem \ref{thm:upper-bound}}

The proof of Theorem \ref{thm:upper-bound} proceeds by starting from the result of Molchan on the finiteness of moments of $Y$ of negative order and using this information to get better bounds on the decay of the Laplace transform of $Y$ near infinity. This procedure is then iterated, giving better and better estimates. We state the following proposition, which adapts the methods of Liu \cite{li01} and Barral \cite{ba04} to this iteration procedure, as a result of its own.

\begin{prop}\label{prop:upper-bound}
Suppose $W$ and $Y$ satisfy \eqref{eq:smoothing-inequality} and \eqref{eq:w-small-tail} and further that for some $\alpha \in [1,\gamma[$ there exist constants $t_\alpha, c_\alpha > 0$ such that the estimate \eqref{eq:main-estimate} is satisfied. Then for any $\alpha' \in \left[\alpha,\alpha+(\gamma-\alpha)/(\gamma+1)\right]$ there exist constants $t_{\alpha'}, c_{\alpha'} > 0$ such that \eqref{eq:main-estimate} holds with $(\alpha', t_{\alpha'}, c_{\alpha'})$ in place of $(\alpha, t_\alpha, c_\alpha)$.
\end{prop}

\begin{proof}
Throughout the proof we will for brevity denote $\phi(t) = \E e^{-t Y}$ and $\psi(t) = \E \, \phi(tW)^2$. Since $\Prob(W=0) = \Prob(Y=0) = 0$, the functions $\phi$ and $\psi$ are both strictly decreasing homeomorphisms of $[0,\infty[$ onto $]0,1]$. The assumption that $\phi$ satisfies \eqref{eq:main-estimate} for $\alpha \in [1,\gamma[$ implies that $\psi$ satisfies
\begin{align*}
\psi(t) = \E \phi(tW)^2 & \leq \Prob( W \leq t^{-1/2} ) + \phi(t^{1/2})^2 \\
& \leq e^{-\frac{c}{2^\gamma} (\log t)^\gamma} + e^{-\frac{2 c_\alpha}{2^\alpha} (\log t)^\alpha}
\end{align*}
for all $t \geq t_\alpha$. It follows that there exist constants $\tilde t_\alpha > 0$ and $C_\alpha > 0$ such that
\begin{equation}\label{eq:psi-decay-rate}
\psi(t) \leq \exp\left(-C_\alpha (\log t)^\alpha \right) \quad \textrm{for all } t \geq \tilde t_\alpha.
\end{equation}

The main estimate of the proof is derived next. We use the distributional inequality \eqref{eq:smoothing-inequality} and the Cauchy--Schwarz inequality to deduce
\begin{align*}
\phi(t) = \E e^{-t Y} & \leq \E e^{-t (W_0 Y_0 + W_1 Y_1)} = \E \phi(t W_0) \phi(t W_1) \\
& \leq \sqrt{\E \phi(t W_0)^2} \sqrt{\E \phi(t W_1)^2} = \E \phi(t W)^2 \\
& = \psi(t).
\end{align*}
It follows that for all $t \geq t' > 0$ we have
\begin{equation}\label{eq:first-step}
\phi(t)^2 \leq \psi(t)^2 \leq \psi(t') \psi(t).
\end{equation}
Using this observation with $\tilde t = t W$ and $\tilde t' = t^{1/2} \geq 1$, we get the estimate
\begin{align}
\psi(t) & = \E \phi(tW)^2 \leq \Prob\left( W \leq t^{-1/2} \right) + \psi(t^{1/2}) \E \psi(tW) \ind_{\left\{ W > t^{-1/2} \right\}} \nonumber \\
& = \Prob\left( W_1 \leq t^{-1/2} \right) + \psi(t^{1/2}) \E \phi(t W_1 W_2)^2 \ind_{\left\{ W_1 > t^{-1/2} \right\}} \label{eq:second-step},
\end{align}
where $W_1$ and $W_2$ denote independent copies of $W$. We plug this estimate into \eqref{eq:first-step} to obtain, for all $t \geq 1$,
$$
\phi(t)^2 \leq \psi(t^{1/2}) \Prob\left( W_1 \leq t^{-1/2} \right) + \psi(t^{1/2})^2 \E \phi(t W_1 W_2)^2 \ind_{\left\{ W_1 > t^{-1/2} \right\}}.
$$
Moreover, by using \eqref{eq:first-step} as above, for an arbitrary positive random variable $V$ which is independent of $W' \eqlaw W$ we have
\begin{align}
\E \phi(t V W')^2 \ind_{\{V > t^{-1/2}\}} & \leq \Prob\left( V > t^{-1/2}, V W' \leq t^{-1/2} \right) \nonumber \\
& \:\:\:\:\: + \psi(t^{1/2}) \E \phi(t V W' W'')^2 \ind_{\{V > t^{1/2}, V W' > t^{1/2}\}}, \label{eq:iteration-step}
\end{align}
where $W'' \eqlaw W$ is independent of $V$ and $W'$. Let $(W_n)_{n \geq 1}$ be an i.i.d. sequence of copies of $W$ and define the stopping time
$$
\tau_t = \inf \left\{ n \geq 1 \,\big\vert\, W_1 W_2 \dots W_n \leq t^{-1/2} \right\}.
$$
Using \eqref{eq:iteration-step} iteratively for $V = W_1$, $V = W_1 W_2$, \dots we obtain, for any $n \in \N$ and all $t \geq 1$,
\begin{equation}\label{eq:phi2-estimate}
\phi(t)^2 \leq \sum_{k=1}^n \psi(t^{1/2})^k \Prob(\tau_t = k) + \psi(t^{1/2})^{n+1} \E \phi(t W_1 W_2 \dots W_{n+1})^2 \ind_{\{\tau_t > n\}}.
\end{equation}

It remains to compute that the estimate \eqref{eq:phi2-estimate} indeed gives \eqref{eq:main-estimate} for $\alpha' \in [\alpha,\alpha+(\gamma-\alpha)/(\gamma+1)]$. For any $t \geq 1$, the probability $\Prob(\tau_t = k)$ may be estimated by
\begin{align*}
\Prob(\tau_t = k) & \leq \Prob( W_1 W_2 \dots W_k \leq t^{-1/2} ) \\
& \leq k \Prob( W \leq t^{-1/2k} ) \\
& \leq k \exp\left( -\frac{c}{2^\gamma} k^{-\gamma} (\log t)^\gamma \right)
\end{align*}
as long as
\begin{equation}\label{eq:t-k-condition}
t^{-\frac{1}{2k}} \leq x' \quad \Longleftarrow \quad k \leq 2 \frac{\log t}{-\log x'}.
\end{equation}
The decay rate \eqref{eq:psi-decay-rate} derived for $\psi$ in turn gives, for $t \geq \tilde t_\alpha^2$,
$$
\psi(s^{1/2})^k \leq \exp\left(-\frac{C_\alpha}{2^\alpha} \,k\, (\log t)^\alpha \right)
$$
and therefore the terms of the sum in \eqref{eq:phi2-estimate} may be estimated by
\begin{align}
\psi(t^{1/2})^k \Prob(\tau_t = k) & \leq k \exp\left( -\frac{C_\alpha}{2^\alpha} \,k\, (\log t)^\alpha - \frac{c}{2^\gamma} k^{-\gamma} (\log t)^\gamma \right) \nonumber \\
& =: k \exp(-f_t(k)) \label{eq:term-estimate}
\end{align}
for $t \geq \tilde t_\alpha^2$ and $k \in \N$ satisfying \eqref{eq:t-k-condition}. Finding the minimum of $k \mapsto f_t(k)$ on $]0,\infty[$ is easy: the zero of the derivative is at
$$
k_0 = \left( \frac{c \gamma}{C_\alpha} 2^{\alpha-\gamma} \right)^\frac{1}{\gamma+1} (\log t)^\frac{\gamma-\alpha}{\gamma+1},
$$
so for $t \geq \tilde t_\alpha^2$ we have
\begin{align*}
f_t(k) & \geq \frac{C_\alpha}{2^\alpha} \,k_0\, (\log t)^\alpha + \frac{c}{2^\gamma} k_0^{-\gamma} (\log t)^\gamma \\
& = C' (\log t)^{\alpha + \frac{\gamma-\alpha}{\gamma+1}} + C'' (\log t)^{\gamma - \gamma \frac{\gamma-\alpha}{\gamma+1}}, \\
& = C (\log t)^{\alpha + \frac{\gamma-\alpha}{\gamma+1}}
\end{align*}
where the constants $C', C'' > 0$ and $C = C'+C''$ only depend on the constants $\gamma, c, \alpha$ and $C_\alpha$. Plugging this into \eqref{eq:term-estimate} gives
\begin{equation}\label{eq:term-estimate-final}
\psi(t^{1/2})^k \Prob(\tau_t = k) \leq k \exp\left(-C (\log t)^{\alpha + \frac{\gamma-\alpha}{\gamma+1}} \right)
\end{equation}
for $t \geq \tilde t_\alpha^2$ and $k \in \N$ satisfying \eqref{eq:t-k-condition}. Next we estimate the final term in \eqref{eq:phi2-estimate} and choose the value of $n$. It is enough to use the crude estimates $\phi \leq 1$ and $\Prob(\tau_t > n) \leq 1$ and the decay rate \eqref{eq:psi-decay-rate} to get
\begin{align}
\psi(t^{1/2})^{n+1} \E \phi(t W_1 W_2 \dots W_{n+1})^2 \ind_{\{\tau_t > n\}} & \leq \psi(t^{1/2})^n \nonumber \\
& \leq \exp\left(-C_\alpha n (\log t)^\alpha\right) \label{eq:last-term-estimate}
\end{align}
for $t \geq \tilde t_\alpha^2$. Choosing $n = \left\lceil (\log t)^\frac{\gamma-\alpha}{\gamma+1} \right\rceil$ in \eqref{eq:phi2-estimate}, we see that \eqref{eq:t-k-condition} is satisfied for all sufficiently large $t$ and $k \leq n$, so by \eqref{eq:term-estimate-final} and \eqref{eq:last-term-estimate} we have shown that there exists a $\hat t \geq \tilde t^2$ such that for all  $t \geq \hat t$
\begin{align*}
\phi(t)^2 & \leq \sum_{k=1}^n k \exp\left(-C (\log t)^{\alpha + \frac{\gamma-\alpha}{\gamma+1}} \right) + \exp\left(-C_\alpha n (\log t)^\alpha \right) \\
& \leq \frac{1}{2} \left( \left\lceil (\log t)^\frac{\gamma-\alpha}{\gamma+1} \right\rceil + 1 \right)^2 \exp\left(-C (\log t)^{\alpha + \frac{\gamma-\alpha}{\gamma+1}} \right) + \exp\left(-C_\alpha (\log t)^{\alpha + \frac{\gamma-\alpha}{\gamma+1}} \right).
\end{align*}
For any $\alpha' \in \left[\alpha,\alpha+(\gamma-\alpha)/(\gamma+1)\right]$, the prefactor in the first term above may be absorbed in order to obtain the desired constants $c_{\alpha'}, t_{\alpha'} > 0$ for which we have the estimate
$$
\phi(t) \leq \exp\left(-c_{\alpha'} (\log t)^{\alpha'}\right)
$$
for all $t \geq t_{\alpha'}$. The proof is complete.
\end{proof}

Theorem \ref{thm:upper-bound} now follows by iteration.

\begin{proof}[Proof of Theorem \ref{thm:upper-bound}.]
Let $Y$ and $W$ satisfy \eqref{eq:smoothing-inequality} and \eqref{eq:w-small-tail}. It is a well-known result of Molchan\footnote{To be exact, Molchan considers only Mandelbrot cascades, but for instance the proof given by Liu \cite{li01} in a more general situation also works in our more restricted setup with only cosmetic modifications.} \cite{mo96} that, since $\E W^{-q} < \infty$ for all $q > 0$, for any $q > 0$ there exists a constant $C_q > 0$ for which $\phi(t) \leq C_q t^q$ for all sufficiently large $t$. It follows that \eqref{eq:main-estimate} holds, for some constants $t_1, c_1 > 0$, for $\alpha = 1$.

Let $\alpha_0 = 1$. By Proposition \ref{prop:upper-bound}, for any $\alpha \in [1,1+(\gamma-1)/(\gamma+1)]$ we may find the constants $t_\alpha,c_\alpha > 0$ for which \eqref{eq:main-estimate} holds. Denote $\alpha_1 = 1 + (\gamma-1)/(\gamma+1)$ and generally $\alpha_n = \alpha_{n-1} + (\gamma-\alpha_{n-1})/(\gamma+1)$ for $n \in \N$. Suppose that we have shown that for some $n \in \N$, for all $\alpha \in [1,\alpha_n]$ there exist constants $t_\alpha,c_\alpha > 0$ such that \eqref{eq:main-estimate} holds. Then if $\alpha \in [\alpha_n,\alpha_{n+1}]$, 
the existence of the constants $t_\alpha,c_\alpha > 0$ for which \eqref{eq:main-estimate} holds follows Proposition \ref{prop:upper-bound}. By induction, we see that \eqref{eq:main-estimate} holds for all $\alpha \in [1, \lim_{n \to \infty} \alpha_n[$. But since
$$
\alpha_0 = 1 \quad \textrm{and} \quad \alpha_n = \frac{\gamma}{\gamma+1} \alpha_{n-1} + \frac{\gamma}{\gamma+1} \quad \textrm{for } n \in \N,
$$
the general term is given explicitly by $\alpha_n = \sum_{k=1}^{n-1} \left( \frac{\gamma}{\gamma+1} \right)^k + 2 \left( \frac{\gamma}{\gamma+1} \right)^n$ for $n > 0$,
from which it is immediate that
$$
\lim_{n \to \infty} \alpha_n = \sum_{k=1}^\infty \left( \frac{\gamma}{\gamma+1} \right)^k = \gamma.
$$
\end{proof}

Theorem \ref{thm:upper-bound} gives an upper bound for the Laplace transform of $Y$. In order to prove Theorem \ref{thm:cascades} we also need the following proposition, again adapting the earlier method of Liu \cite{li01}, to give a lower bound.

\begin{prop}\label{prop:lower-bound}
Suppose the positive random variable $W$ satisfies, for some $\gamma > 1$ and $x' \in ]0,1[$,
\begin{equation}\label{eq:w-small-tail-reversed}
\Prob(W \leq x) \geq \exp\left( - c (-\log x)^\gamma \right) \quad \textrm{for all } x \leq x'
\end{equation}
and that $Y$ satisfies
\begin{equation}\label{eq:smoothing-inequality-reversed}
Y \preceq W_0 Y_0 + W_1 Y_1,
\end{equation}
where $(W_0,W_1)$ and $(Y_0,Y_1)$ are independent pairs of copies of $W$ and $Y$, independent of each other. Then there exist constants $t_\gamma,c_\gamma > 0$ such that
\begin{equation}\label{eq:main-estimate-reversed}
\phi(t) \geq \exp\left( -c_\gamma (\log t)^\gamma \right) \quad \textrm{for all } t \geq t_\gamma.
\end{equation}
\end{prop}

\begin{proof}
By \eqref{eq:smoothing-inequality-reversed}, for all $t \geq 1$ we have
\begin{align*}
\phi(t) = \E e^{-t W_0 Y_0 - t W_1 Y_1} & = \E \phi(t W_0) \phi(t W_1) = \left( \E \phi(tW) \right)^2 \\
& \geq \Prob(W \leq t^{-1/2})^2 \phi(t^{1/2})^2.
\end{align*}
Iterating this estimate, for all $n \in \N$ we have
\begin{equation}\label{eq:lower-bound-iterated}
\phi(t) \geq \left( \prod_{k=1}^n \Prob\left( W \leq t^{-2^{-k}} \right)^{2^k} \right) \phi\left(t^{2^{-n}}\right)^{2^n}
\end{equation}
Let $n$ be the greatest integer such that $t^{-2^{-n}} \leq x'$, i.e. the unique integer such that
$$
t^{-2^{-n}} \leq x' < t^{-2^{-n-1}} \quad \Longleftrightarrow \quad 2^n \leq \frac{\log t}{-\log x'} < 2^{n+1}.
$$
With this choice \eqref{eq:lower-bound-iterated} gives
\begin{align*}
\phi(t) & \geq \left( \prod_{k=1}^n \Prob\left( W \leq t^{-2^{-k}} \right)^{2^k} \right) \phi(1/x')^{\frac{\log t}{-\log x'}} \\
& \geq \exp \left( -c \sum_{k=1}^n 2^k \left( \log t^{2^{-k}} \right)^\gamma - \frac{\log \phi(1/x')}{\log x'} \log t \right) \\
& = \exp \left( -c \left( \sum_{k=1}^n 2^{-(\gamma-1)k} \right) (\log t)^\gamma - \frac{\log \phi(1/x')}{\log x'} \log t \right) \\
& \geq \exp \left( - \frac{c}{2^{\gamma-1}-1} (\log t)^\gamma - \frac{\log \phi(1/x')}{\log x'} \log t \right).
\end{align*}
Since $\gamma > 1$, it is now clear that there exist constants $t_\gamma,c_\gamma > 0$ such that
$$
\phi(t) \geq \exp\left( -c_\gamma (\log t)^\gamma \right) \quad \textrm{for } t \geq t_\gamma.
$$
\end{proof}

\section{Applications to Mandelbrot cascades and lognormal multiplicative chaos}\label{se:applications}

We present applications of the preceding analysis to two cases of interest. Theorem \ref{thm:cascades}, the first application, concerns fixed points of the smoothing transform and thus applies to Mandelbrot cascades.

\begin{proof}[Proof of Theorem \ref{thm:cascades}.]
Let $W$ be a positive random variable satisfying
$$
\lim_{x \to 0} \frac{\log \log 1/\Prob(W \leq x)}{\log \log 1/x} = \gamma > 1.
$$
For any $\gamma_-,\gamma_+$ such that $1 < \gamma_- < \gamma < \gamma_+$ there exists a $x' \in ]0,1[$ for which
$$
\exp\left( -(\log 1/x)^{\gamma_+} \right) \leq \Prob(W \leq x) \leq \exp\left( -(\log 1/x)^{\gamma_-} \right) \quad \textrm{for all } 0 < x \leq x'.
$$
From Theorem \ref{thm:upper-bound} and Proposition \ref{prop:lower-bound} it follows that there exist constants $t',c_{\gamma_-},c_{\gamma_+} > 0$ such that
$$
\exp\left( -c_{\gamma_+} (\log t)^{\gamma_+} \right) \leq \phi(t) \leq \exp\left( -c_{\gamma_-} (\log t)^{\gamma_-} \right) \quad \textrm{for all } t \geq t'
$$
or equivalently
$$
\log c_{\gamma_-} + \gamma_- \log \log t \leq \log \log 1/\phi(t) \leq \log c_{\gamma_+} + \gamma_+ \log \log t \quad \textrm{for all } t \geq t'.
$$
It follows that
$$
\gamma_- \leq \liminf_{t \to \infty} \frac{\log \log 1/\phi(t)}{\log \log 1/t} \leq \limsup_{t \to \infty} \frac{\log \log 1/\phi(t)}{\log \log 1/t} \leq \gamma_+.
$$
Since $\gamma_- < \gamma$ and $\gamma_+ > \gamma$ are arbitrary, this implies the claim.
\end{proof}

Our estimate for the decay of the Laplace transform $\E e^{-t Y}$ as $t \to \infty$ results in an estimate for the probabilities $\Prob(Y \leq x)$ as $x \to 0$, as stated in Corollary \ref{cor:cascades}.

\begin{proof}[Proof of Corollary \ref{cor:cascades}.]
An sufficient upper bound for $\Prob(Y \leq x)$ is given by Markov's inequality:
$$
\Prob\left(Y \leq x\right) \leq e \E e^{-\frac{1}{x} Y}.
$$
By Theorem \ref{thm:cascades}, for any $\gamma_- < \gamma$ we have
\begin{equation}\label{eq:probabilities-upper}
\liminf_{x \to 0} \frac{\log \log 1/\Prob\left(Y \leq x\right)}{\log \log 1/x} \geq \liminf_{x \to 0} \frac{\log\left(-1 + \log 1/\E e^{-\frac{1}{x} Y}\right)}{\log \log 1/x} \geq \gamma_-.
\end{equation}

For the lower bound we use the estimate
\begin{align*}
\E e^{-x^{-2} Y} & = \E e^{-x^{-2} Y} \ind_{\{Y \leq x\}} + \E e^{-x^{-2} Y} \ind_{\{Y > x\}} \\
& \leq \Prob(Y \leq x) + e^{-x^{-1}}
\end{align*}
By Theorem \ref{thm:cascades}, for any $\gamma_+ > \gamma$, for all $x > 0$ small enough we have
$$
\E e^{-x^{-2} Y} \geq e^{-(\log x^{-2})^{\gamma_+}} = e^{-2^{\gamma_+} (\log 1/x)^{\gamma_+}},
$$
implying that
$$
\Prob(Y \leq x) \geq \E e^{-x^{-2} Y} - e^{-x^{-1}} \geq \frac{1}{2} e^{-2^{\gamma_+} (\log 1/x)^{\gamma_+}}
$$
for all $x > 0$ small enough. Thus
$$
\log \log 1/\Prob(Y \leq x) \leq \log \left(\log 2 + 2^{\gamma_+} (\log 1/x)^{\gamma_+}\right),
$$
which implies
\begin{equation}\label{eq:probabilities-lower}
\limsup_{x \to 0} \frac{\log \log 1/\Prob(Y \leq x)}{\log \log 1/x} \leq \gamma_+.
\end{equation}
Together the bounds \eqref{eq:probabilities-upper} and \eqref{eq:probabilities-lower} imply the claim.
\end{proof}


We then consider an application of Theorem \ref{thm:upper-bound} to lognormal $\star$-scale invariant multiplicative chaos. We refer the reader to the recent survey \cite{rhva13} of Rhodes and Vargas for an introduction to this class of random measures, and give here an application to a particular random measure on $\R$ that is both simple enough to have a compact definition, yet which illustrates the range of applicability of Theorem \ref{thm:upper-bound}.

\begin{defn}
Let $M$ be a positive random measure on $\R^d$ that satisfies the distributional scaling relation
$$
\left( M(A) \right)_{A \in \Bcal(\R^d)} \eqlaw \left( \int_A e^{\omega_\varepsilon(x)} M^\varepsilon (\d x) \right)_{A \in \Bcal(\R^d)}, \quad \omega \perp M^\varepsilon
$$
where $M^\varepsilon$ is a positive random measure on $\R^d$ with the law given by
$$
\left( M^\varepsilon(A) \right)_{A \in \Bcal(\R^d)} \eqlaw \left( M(\epsilon^{-1} A) \right)_{A \in \Bcal(\R^d)}
$$
and $\left( \omega_\varepsilon(x) \right)_{x \in \R^d, \varepsilon \in ]0,1]}$ is a Gaussian process.

If $M$ satisfies the scaling relation above for a given Gaussian process $\omega$, we say that $M$ is lognormal $\star$-scale invariant.
\end{defn}

Under certain conditions on the process $\omega$, it has been shown that a nontrivial lognormal $\star$-scale invariant $M$ can be constructed as lognormal multiplicative chaos, a construction of a positive random measure given by Kahane in \cite{ka85} and recently extended to the so-called critical case by Duplantier, Rhodes, Sheffield and Vargas \cite{drsv12-1,drsv12-2}. Conversely Rhodes, Sohier and Vargas \cite{rhsova12} have shown that Kahane's construction gives (essentially) all the stationary lognormal $\star$-scale invariant random measures such that the masses of open sets have finite moments of order $1+\delta$ for some $\delta > 0$. Constructions of lognormal $\star$-scale invariant with infinite expectations of masses of open sets have been given, but their (essential) uniqueness has yet to be proven.

We then briefly summarize the construction of the lognormal $\star$-scale invariant random measure to be considered below and refer the reader to \cite{bknsw13} for a more detailed exposition for this kind of a construction. We take $d = 1$ and $\omega$ as defined by
$$
\omega_\varepsilon(x) = \beta X_\varepsilon(x) - \frac{\beta^2}{2} \E X_\varepsilon(x)^2,
$$
where $\beta > 0$ is a parameter and $X$ is a centered Gaussian process with the covariance
$$
\E X_\varepsilon(x) X_{\varepsilon'}(y) = \begin{cases}
\log \frac{1}{\varepsilon \vee \varepsilon'} - \left(\frac{1}{\varepsilon \vee \varepsilon'}-1\right) |x-y|, & |x-y| < \varepsilon \vee \varepsilon' \\
\log \frac{1}{|x-y|} - 1 + |x-y|, & \varepsilon \vee \varepsilon' \leq |x-y| \leq 1 \\
0, & 1 < |x-y|\\
\end{cases}.
$$
The field $X$ may be visualized by considering white noise $W$ on the upper half-plane with control measure $\d \lambda = \d x \d y/y^2$, and integrating $W$ on the truncated triangles
$$
T_\varepsilon(x) = \left\{ (x',y') \in \R \times ]0,\infty[ \,\big\vert\, \max(2|x-x'|) \leq y' \leq 1 \right\},
$$
i.e. taking $X_\varepsilon(x) = W(T_\varepsilon(x))$. The $\star$-scale invariant random measures associated to $\omega$ defined this way may be constructed as
$$
M(\d x) = \lim_{\varepsilon \to 0} e^{\beta X_\varepsilon(x) - \frac{\beta^2}{2} \E X_\varepsilon(x)^2} \d x,
$$
and it can be shown that, restricting the measures involved to an arbitrary bounded interval, the limit $M$ exists almost surely in the sense of weak convergence of measures. For $\beta < \sqrt{2}$ the limit $M$ is almost surely positive on any interval but for $\beta \geq \sqrt{2}$ the limit is almost surely null. In the critical case $\beta = \sqrt{2}$ a $\star$-scale invariant measure is obtained as a limit in probability by renormalizing the density by $\sqrt{\log 1/\varepsilon}$, see \cite{drsv12-1,drsv12-2}, and it is expected that for $\beta > \sqrt{2}$ another deterministic renormalization will in the distributional limit result in a $\star$-scale invariant measure though this is yet to be proven. The theorem below applies to any of these measures (i.e. for any $\beta > 0$), but for notational convenience we state it for $\beta \in ]0,\sqrt{2}[$.
 
\begin{thm}\label{thm:chaos}
Let $M$ be the lognormal $\star$-scale invariant multiplicative chaos measure on $\R$ defined above for some choice of the parameter $\beta \in ]0,\sqrt{2}[$. The Laplace transform $\phi(t) = \E \exp(-t M([0,1]))$ of the mass of the unit interval satisfies
$$
\lim_{t \to \infty} \frac{\log \log 1/\phi(t)}{\log \log t} = 2.
$$
\end{thm}
The following estimate is derived from Theorem \ref{thm:cascades} in the same way as Corollary \ref{cor:cascades} from Theorem \ref{thm:cascades}.
\begin{cor}\label{cor:chaos}
Let $M$ as in Theorem \ref{thm:chaos}. Then also
$$
\lim_{x \to 0} \frac{\log \log 1/\Prob(M([0,1]) \leq x)}{\log \log 1/x} = 2.
$$
\end{cor}

\vspace{0.3cm}

\begin{proof}[Proof of Theorem \ref{thm:chaos}.]
The required lower bound for the Laplace transform follows from a comparison to a lognormal Mandelbrot cascade using Kahane's convexity inequalities and Proposition \ref{prop:lower-bound}. This kind of a comparison between multiplicative chaos and Mandelbrot cascades has been utilized already by Kahane \cite{ka85} and more recently, for example, in \cite{bknsw13,drsv12-1} so we will only sketch the argument. One constructs a Gaussian field $(Y_\varepsilon(x))_{x \in [0,1], \varepsilon \in ]0,1]}$ in such a way that the covariance of $Y$ is dominated by the covariance of $X$, i.e. $\E Y_\varepsilon(x) Y_\varepsilon(y) \leq \E X_\varepsilon(x) X_\varepsilon(y)$ for all $0 \leq x,y \leq 1$ and $0 < \varepsilon \leq 1$, and that the measure
$$
\tilde M(\d x) = \lim_{\varepsilon \to 0} e^{\beta Y_\varepsilon(x) - (\beta^2/2) \E Y_\varepsilon(x)^2} \d x
$$
coincides with a lognormal Mandelbrot cascade measure on $[0,1]$, multiplied by an independent lognormal factor. Kahane's convexity inequality then states that one has
$$
\E F(\tilde M([0,1])) \leq \E F(M([0,1]))
$$
for any convex $F:[0,\infty[ \to [0,\infty[$ that grows at most polynomially near $\infty$. Especially, by Proposition \ref{prop:lower-bound} there exist constants $t_2,c_2 > 0$ such that
\begin{equation}\label{eq:chaos-laplace-lower-bound}
\exp\left(-c_2 (\log t)^2\right) \leq \E e^{-t \tilde M([0,1])} \leq \E e^{-t M([0,1])}
\end{equation}
for all $t \geq t_2$, since a lognormal variable $W$ satisfies
$$
\Prob(W \leq w) \geq \exp\left(-c (\log 1/w)^2\right)
$$
for small enough $w > 0$, with some constant $c > 0$ depending on the mean and variance of $\log W$.

The upper bound for the Laplace transform is derived by using $\star$-scale invariance and Theorem \ref{thm:upper-bound}. By $\star$-scale invariance of $M$, we may write
\begin{align*}
Y := M([0,1]) & = M([0,1/3]) + M(]1/3,2/3[) + M([2/3,1]) \\
& \geq M([0,1/3]) +  M([2/3,1]) \\
& = \int_0^{1/3} e^{\omega_{1/3}(x)} M^{1/3}(\d x) + \int_{2/3}^1 e^{\omega_{1/3}(x)} M^{1/3}(\d x) \\
& \geq W_0 Y_0 + W_1 Y_1,
\end{align*}
where we have denoted $Y_0 = M^{1/3}([0,1/3])$, $Y_1 = M^{1/3}([2/3,1])$ and
$$
W_0 = \frac{1}{3} \inf_{x \in [0,1/3]} e^{\omega_{1/3}(x)} \quad \textrm{and} \quad W_1 = \frac{1}{3} \inf_{x \in [2/3,1]} e^{\omega_{1/3}(x)}.
$$
The pairs $(W_0,W_1)$ and $(Y_0,Y_1)$ are independent of each other, since $\omega$ and $M^{1/3}$ are, by $\star$-scale invariance, independent. The stationarity of $(X_{1/3}(x))_{x \in [0,1]}$ implies that $W_0 \eqlaw W_1$. It follows from our construction that
\begin{gather*}
M^{1/3}([0,1/3]) = \lim_{\varepsilon \to 0} \int_0^{1/3} e^{\beta \left(X_\varepsilon(x) - X_{1/3}(x)\right) - \frac{\beta^2}{2} \E X_\varepsilon(x)^2 + \frac{\beta^2}{2} \E X_{1/3}(x)^2} \d x, \\
M^{1/3}([2/3,1]) = \lim_{\varepsilon \to 0} \int_{2/3}^1 e^{\beta \left(X_\varepsilon(x) - X_{1/3}(x)\right) - \frac{\beta^2}{2} \E X_\varepsilon(x)^2 + \frac{\beta^2}{2} \E X_{1/3}(x)^2} \d x
\end{gather*}
and from the covariance structure of $X$ we may check that the fields
$$
\left(X_\varepsilon(x) - X_{1/3}(x)\right)_{x \in [0,1/3], \varepsilon < 1/3} \quad \textrm{and} \quad \left(X_\varepsilon(x) - X_{1/3}(x)\right)_{x \in [2/3,1], \varepsilon < 1/3}
$$
are independent of each other. It follows that $Y_0$ and $Y_1$ are independent. Finally, from $\star$-scale invariance we see that $Y_0 \eqlaw Y_1 \eqlaw Y$. To apply Theorem \ref{thm:upper-bound}, all that remains is to bound $\Prob(W_0 \leq w)$.

The value of $W_0$ is determined by the minimum of the centered Gaussian field $X_{1/3}$ on $[0,1/3]$. Good bounds for the probabilities of extremal values of a Gaussian process being large are given by the Borell--Tsirelson--Ibrahimov--Sudakov inequality (see e.g. \cite{ad90}), though the following bound can certainly be obtained from weaker results. In our case, since the covariance $\E X_{1/3}(x) X_{1/3}(y)$ is bounded and Lipschitz on $(x,y) \in [0,1/3]^2$, there exist constants $c',C,A > 0$ such that for all $a \geq A$ we have
$$
\Prob\left( \inf_{x \in [0,1/3]} X_{1/3}(x) \leq -a \right) \leq C \exp\left(-c' a^2\right).
$$
A short computation shows that this is equivalent to
$$
\Prob(W_0 \leq w) \leq C \exp\left(-c' \left( -\frac{1}{\beta}\log w - \left(\frac{\beta}{2}+\frac{1}{\beta}\right) \log 3 \right)^2\right)
$$
for sufficiently small values of $w > 0$. For $0 < c < c/\beta^2$ we thus have $w' > 0$ such that
$$
\Prob(W_0 \leq w) \leq \exp\left(-c(-\log w)^2\right)
$$
for all $0 < w \leq w'$. By Theorem \ref{thm:upper-bound}, we conclude that for any $1 \leq \alpha < 2$ there exist constants $t_\alpha,c_\alpha > 0$ such that
\begin{equation}\label{eq:chaos-laplace-upper-bound}
\E e^{-t M([0,1])} = \E e^{-tY} \leq \exp\left( -c_\alpha (\log t)^\alpha \right) \quad \textrm{for all } t \geq t_\alpha.
\end{equation}

The theorem follows from \eqref{eq:chaos-laplace-lower-bound} and \eqref{eq:chaos-laplace-upper-bound} just as in the proof of Theorem \ref{thm:cascades}.
\end{proof}


\begin{thebibliography}{bknsw13}
\bibitem{ad90} R. Adler: An Introduction to Continuity, Extrema and Related Topics for General Gaussian Processes, Lecture Notes - Monograph Series, Institute Mathematical Statistics, Hayward, CA. (1990) (Errata: http://webee.technion.ac.il/people/adler/borell.pdf)

\bibitem{aish11} E. A\"id\'ekon and Z. Shi: The Seneta-Heyde scaling for the branching random walk, arXiv:1102.0217.

\bibitem{ba04} J. Barral: Techniques for the study of infinite products of independent random functions (Random multiplicative multifractal measures. III). Fractal geometry and applications: a jubilee of Benoît Mandelbrot (Proc. Sympos. Pure Math. 72), Part 2, 53–90. (2004)

\bibitem{bknsw13} J. Barral, A. Kupiainen, M. Nikula, E. Saksman, C. Webb: Basic properties of critical lognormal multiplicative chaos. arXiv:1303.4548

\bibitem{drsv12-1} B. Duplantier, R. Rhodes, S. Sheffield, V. Vargas: Critical Gaussian Multiplicative Chaos: Convergence of the Derivative Martingale. arXiv:1206.1671

\bibitem{drsv12-2} B. Duplantier, R. Rhodes, S. Sheffield, V. Vargas: Renormalization of Critical Gaussian Multiplicative Chaos and KPZ formula. arXiv:1212:0529

\bibitem{duli83} R. Durrett, T. Liggett: Fixed points of the smoothing transformation. Z. Wahrsch. Verw. Gebiete 64 (1983), no. 3, 275--301.

\bibitem{howa92} R. Holley, E.C. Waymire: Multifractal dimensions and scaling exponents for strongly bounded random cascades. Ann. Appl. Probab. 2, no. 4, 819–845. (1992)

\bibitem{hu13} Y. Hu: How big is the minimum of a branching random walk? arXiv:1305.6448

\bibitem{li01} Q. Liu: Asymptotic properties and absolute continuity of laws stable by random weighted mean. Stochastic Process. Appl. 95, no. 1, 83--107. (2001)

\bibitem{ka85} J.-P. Kahane: Sur le chaos multiplicatif. Ann. Sci. Math. Québec 9, no. 2, 105--150. (1985)

\bibitem{kape76} J.-P. Kahane, J. Peyri\`ere, Sur certaines martingales de B.~Mandelbrot. Adv.\ Math. 22, 131--145. (1976)

\bibitem{ma11} T. Madaule: Convergence in law for the branching random walk seen from its tip, arXiv:1107.2543.

\bibitem{ma72} B.B Mandelbrot: Possible refinement of the lognormal hypothesis concerning the distribution of energy in intermittent turbulence, Statistical Models and Turbulence. In Rosenblatt, M. and Atta, C.V. ed. Lectures Notes in Physics. 12, 333–351. Springer–Verlag, New York. (1972)

\bibitem{ma74} B.B. Mandelbrot: Intermittent turbulence in self-similar cascades, divergence of high moments and dimension of the carrier. J. Fluid. Mech. 62, 331--358. (1974)

\bibitem{mo96} G. Molchan: Scaling exponents and multifractal dimensions for independent random cascades. Comm. Math. Phys. 179, no. 3, 681--702. (1996)

\bibitem{os09} D. Ostrovsky: Mellin Transform of the Limit Lognormal Distribution. Comm. Math. Phys. 288, 287--310. (2009)

\bibitem{rhsova12} R. Rhodes, J. Sohier, V. Vargas: Star-scale invariant random measures. arXiv:1201.5219

\bibitem{rhva13} R. Rhodes, V. Vargas: Gaussian multiplicative chaos and applications: a review. arXiv:1305.6221

\bibitem{we11} C. Webb: Exact asymptotics of the freezing transitions of a logarithmically correlated random energy model. J. Stat. Phys, 145, 1595--1619. (2011)
\end{thebibliography}
\end{document}